\documentclass[11pt]{article}
\usepackage{amsmath,amssymb,amsthm,amscd,esint}
\usepackage{url,endnotes,hyperref}
\usepackage{tikz-cd}

\numberwithin{equation}{section}

\newtheorem{theorem}{Theorem}[section]
\newtheorem{lemma}[theorem]{Lemma}
\newtheorem{definition}[theorem]{Definition}

\theoremstyle{corollary}
\newtheorem{corollary}[theorem]{Corollary}
\theoremstyle{conjecture}
\newtheorem{conjecture}{Conjecture}

\theoremstyle{assumption}

\theoremstyle{proposition}
\newtheorem{proposition}[theorem]{Proposition}
\theoremstyle{remark}

\numberwithin{equation}{section}
\everymath{\displaystyle}

\newcommand{\reg}{\operatorname{reg}}

\newcommand{\PP}{\mathbb P}

\newcommand{\Rm}{\operatorname{Rm}}

\newcommand{\R}{\mathbb R}\newcommand{\Q}{\mathbb Q}
\newcommand{\Z}{\mathbb Z}\newcommand{\C}{\mathbb C}
\newcommand{\A}{\mathbb A}

\newcommand{\Ric}{\operatorname{Ric}}

\newcommand{\Bl}{\operatorname{Bl}}

\begin{document}

\title{Collapsed Finite Time Singularities of the Kähler-Ricci Flow on Complex Surfaces are of Type I}
\author{}
\date{}
\maketitle

\vspace{-3em}
\begin{center}
\Large
\begin{tabular}{cc}
Tongxin Xu$^*$ & \qquad Zhenlei Zhang$^\dagger$
\end{tabular}
\end{center}
\vspace{0.1em}

\begin{abstract}
We prove that any finite-time collapsing Kähler-Ricci flow on compact Kähler surfaces develops a Type I singularity. Together with the previous results, this implies that any finite time singularity of the Kähler-Ricci flow on compact Kähler surfaces is of Type I.
\end{abstract}
\maketitle
\tableofcontents

\section{Introduction}

Ricci flow was introduced by Hamilton  as an evolution equation for Riemannian metrics in 1980s \cite{hamilton1982three}. In this paper, we mainly consider the Kähler-Ricci flow. Let $(M,\omega_{0})$ be a compact Kähler manifold of complex
dimension $n$, and let $\omega(t)$ be the solution of the unnormalized
Kähler--Ricci flow
\begin{equation*}
    \frac{\partial}{\partial t}\omega(t)
    =
    -\Ric(\omega(t)),
    \qquad
    \omega(0)=\omega_{0},
\end{equation*}
on its maximal time interval $[0,T)$. Throughout this paper, we assume
that the maximal existence time is finite, namely $T<\infty$. The flow is
said to develop a Type~I singularity at time $T$ if
\begin{equation}\label{eq:type-I}
    \sup_{M}|\Rm(\omega(t))|_{\omega(t)}
    \leq \frac{C}{T-t}
\end{equation}
for some uniform constant $C<\infty$ and every $t\in[0,T)$.

A fundamental problem in the study of finite-time singularities is to
understand the geometry arising under parabolic rescaling. By the work
of Naber and Enders--Müller--Topping, suitable blow-ups based at a
Type~I singular point subconverge to a complete nonflat Ricci shrinker with bounded curvature \cite{Naber,EMT}. In the
Kähler setting, the limit is naturally a Kähler-Ricci shrinker.

A folklore conjecture asserts that finite-time singularities of the
K\"ahler--Ricci flow on compact K\"ahler surfaces are always of Type~I.

\begin{conjecture}\label{conj}
Finite time singularities of Kähler-Ricci flow on compact Kähler surfaces are all type I singularities.
\end{conjecture}

This conjecture may be decomposed into three components.

    (1) Finite time extinction: $\lim\nolimits_{t\nearrow T}\operatorname{diam}(M,\omega(t))=0$, in this case $M$ is a Fano surface with $[\omega_0]= \lambda c_1(M)$ \cite[Theorem 1.4]{TosattiZhang2018FiniteTimeCollapsing}, the flow develops type I singularity and the correspond tangent flow are unique compact Kähler-Ricci shrinker \cite{Tia90,Koi90,WZ04,ST08,TZ07,CW12,BM87,TZ00}.

    (2) Finite time noncollapsing: $\lim\nolimits_{t\nearrow T}\operatorname{Vol}(M,\omega(t))>0$, in this case such a singularity corresponds to contract $(-1)$-curves \cite{SW}, Conlon--Hallgren--Ma prove such singularities are of Type I \cite{CHM25}, the correspondence tangent flow is the unique FIK shrinker \cite{CCD}.

    (3) Finite time collapsing: $\lim\nolimits_{t\nearrow T}\operatorname{Vol}(M,\omega(t))=0$ and $\limsup\nolimits_{t\nearrow T}\operatorname{diam}(M,\omega(t))>0$, in this case $M$ is brational to a ruled surface \cite{SWL}, \cite{XZ,JS} prove the Type I curvature estimate near regular fibers and identify the corresponding tangent flow $\mathbb{P}\times \mathbb{C}$.

In the volume collapsing case, it is equivalent to study the holomorphic $\PP^1$-fibration $p:M\to \Sigma$ over a compact Riemannian surface with $[\omega_0]-2\pi Tc_1(M)=\pi^*[\Sigma]$\cite[Theorem 1.4]{TosattiZhang2018FiniteTimeCollapsing}. Let $\Delta\subseteq \Sigma$ denote the set of points whose fibers are not biholomorphic
to $\PP^1$. We establish the following results:
\begin{theorem}\label{main}
Let $(M,\omega(t))_{t\in[0,T)}$ be a finite time Kähler-Ricci flow on compact Kähler surface $M$. If $\lim\nolimits_{t\nearrow T}\operatorname{Vol}(M,\omega(t))=0$, then $\omega(t)$ develops a finite time Type I singularity at $t=T$.
\end{theorem}
We also identify the correspond tangent flow on singular fiber, the candidate tangent shrinker was constructed by \cite{BCCD}.
\begin{corollary}
For any $x_0\in p^{-1}(\Delta)$ and any sequence $\tau_i\searrow0$, the rescaled flows $(M,\tau_i^{-1}g(T+\tau_i t),J,x_0)_{t\in[-\tau_i^{-1}T,0)}$ subconverge in the pointed $C^{\infty}$ Cheeger-Gromov topology to the BCCD shrinker $\operatorname{Bl}_{p}(\mathbb{P}^1\times \mathbb{C})$.
\end{corollary}

Combining (1), (2), and (3) with our Theorem \ref{main}, Conjecture 1 is solved. When our manuscript was completed, we noticed the work \cite{CCHZ} on arXiv, where they also proved the above results. Our proof is different and was obtained independently.

In higher dimension, this conjecture is false \cite{LTZ24,MT23}. Related results on finite-time singularities can be found in (\cite{SW,JianSongTian2023FiniteTimeSingularities,HJST,CT,LZ26,chen2026k,HZ26},etc.)

As a direct application of Theorem \ref{main}, we obtain the optimal
diameter collapse rate for all fibers of $p:M\to\Sigma$.
\begin{corollary}
Under the same assumptions in Theorem 1.1, there exist uniform constants $c,C>0$ such that
\[
c\sqrt{T-t}
\le \operatorname{diam} (F_q,{g(t)}|_{F_q})
\le C\sqrt{T-t}
\]
for any $q\in\Sigma$ and $0\le t<T$.
\end{corollary}

\paragraph{Outline of the proof.}
The proof has two parts: smoothness of tangent flows and
the Type I estimate.

In the first part, Bamler's compactness theory gives orbifold tangent shrinkers.
Using \cite{HZ26}, we show that the tangent shrinker admits a $\mathbb P^1$-fibration over $\mathbb C$,
with all singularities in the central fiber.
We compare the topology of the local regions near these
singularities with their minimal resolutions.
The rank bound from the original fiber neighborhood and
the blowup description of this fibration give
a contradiction. Thus the tangent shrinker is smooth. The classification \cite{BCCD,LW26} leaves the cylinder and BCCD. We choose scales at which the fiber remains near a
heat-kernel center. A cylinder limit would give a
trivial self-intersection holomorphic curve in the original singular
fiber, contradicting the self-intersections of its
components. This gives one BCCD tangent.
Equality of entropy for tangents of the same kernel
then excludes the cylinder at every other scale.

For the Type I estimate, we first extend the estimate of \cite{XZ} to backward $P^*$-neighborhoods which avoid the singular fibers. We use the secondary
tangent-flow argument of \cite{BCCD}. A local topological
lemma gives smoothness of the secondary limit;
a compact curve and the local cohomology formula
give the contradiction. To obtain the global estimate, the two $(-1)$-curves
of a BCCD tangent deform to the two components of the
original singular fiber. Smooth convergence and bounded
curvature of the model control neighborhoods of this
fiber. Bamler's $P^*$-comparison estimates, together
with the exterior estimate, rule out Type II blowup.
We finally identify the tangent flow on
the singular fibers as the BCCD shrinker.

\paragraph{Use of AI} The authors made substantial use of ChatGPT-6 Astra for the arguments in Section 3.2 when developing the proofs of this paper. The authors have checked all mathematical arguments in the paper and take full responsibility for its content and correctness.

\section{Preliminaries}
\subsection{$\mathbb{P}^1$-fibration and collapsing fibers}

Let $p:M\to \Sigma$ be a holomorphic fibration from a compact Kähler surface onto a compact Riemannian surface $\Sigma$ with general fiber biholomorphic to $\mathbb P^1$. We consider the finite time Kähler-Ricci flow $(M,g(t))_{t\in[0,T)}$ with $[\omega_0]-2\pi Tc_1(M)=p^*[\omega_{\Sigma}]$. The homology class assumptions strongly restricts the singular fibers. 
\begin{lemma}\label{lem:singular-fiber}\cite[Lemma 2.1]{XZ}
Every singular fiber of $p$ is the union of two transversely
intersecting $(-1)$-curves.
\end{lemma}
By the above lemma, we know $M$ is biholomorphic to $\operatorname{Bl}_{p_1,\ldots,p_N}(M_0)$ where $M_0$ is a ruled surface and the points $p_i$ lie on  distinct fibers of $M_0$, for $q\in \Delta$, we set
\begin{equation}\label{eq:two-curves}
 F_q=C_1+C_2,\quad C_j\simeq\PP^1,\quad
 C_1^2=C_2^2=-1,\quad C_1\cdot C_2=1.
\end{equation}
For a disk $D$ containing $q$ and no other singular value, $V=p^{-1}(D)$
therefore has the following topology:
\begin{equation}\label{eq:V}
 \begin{array}{c|c|c}
 q & V & Q_V\text{ on }H_2(V;\Z)\\ \hline
 q\notin\Delta&D\times\PP^1 &[0]\\
 q\in\Delta&\Bl_{\mathrm{pt}}(D\times\PP^1)&
 \begin{pmatrix}-1&1\\1&-1\end{pmatrix}.
 \end{array}
\end{equation}
where $Q_V$ denotes the intersection form. In both cases, $V$ is simply connected, possesses a single end, and satisfies $b_2(V) \leq 2$; moreover, $Q_V$ is negative semidefinite.
Due to $H^2(D;\R)=0$
and $c_1(M)\cdot G=2$. If $\omega_i=\tau_i^{-1}\omega(T-\tau_i)$, then
\begin{equation}\label{eq:local-class}
 [\omega_i]|_V=2\pi c_1(M)|_V,\qquad \int_G\omega_i=4\pi
\end{equation}
for every regular fiber $G$. 

We recall the fiber structure lemma of $\mathbb{P}^1$-fibration. It follows from the standard
contraction argument for rational fibrations; cf.
\cite[Lemmas III.8--III.9]{Beauville}.
We include the proof in the analytic setting.
\begin{lemma}\label{bea}
Let $f:S\to\Delta$ be a proper holomorphic map from a smooth
surface to a disk, with connected fibers and general fiber
$\mathbb P^1$. After shrinking $\Delta$ about $0$, $f$ is obtained
from a smooth $\mathbb P^1$-fibration by finitely many point
blowups over $0$.
\end{lemma}
\begin{proof}
By flatness, the central fiber $F$ satisfies
$F^2=0$ and $K_S\cdot F=-2$.
If $F$ is reducible, every component has negative self-intersection,
and some component $E$ satisfies $K_S\cdot E<0$.
Adjunction then gives
\[
  E^2=K_S\cdot E=-1,\qquad p_a(E)=0.
\]
Contract $E$ by the analytic Castelnuovo theorem.
The map descends and remains proper, with connected fibers
and general fiber $\mathbb P^1$. Repeating this argument, the central fiber becomes $mD$ with
$D$ irreducible. Since
\[
  D^2=0,\qquad -2=mK\cdot D=m(2p_a(D)-2),
\]
we have $m=1$ and $D\simeq\mathbb P^1$.
After shrinking $\Delta$, the new map is a smooth
$\mathbb P^1$-fibration. The original map is obtained from it
by finitely many point blowups over $0$.
\end{proof}

\subsection{Compactness of Ricci flow}
In this section, we review some of Bamler's weak compactness and partial regularity theories that we need \cite{Bam20a,Bam23,Bam20b}. For the Kähler case, we refer to \cite{CHM25}.
\subsubsection{Conjugate heat kernel and pointed Nash entropy}
Let $(M,g(t)_{t\in[0,T)})$ be a Ricci flow on closed manifold. For $(x,t)\in M\times [0,T)$, let $K(x,t,\cdot,\cdot):M\times[0,T)\to (0,\infty)$ denote the conjugate heat kernel based at $(x,t)$. Define probability measures $d\nu_{x,t;s}:=K(x,t;\cdot,s)dg_s$, where $dg_s$ denotes Riemannian volume measure of $(M,g_s)$ for $s\in[0,T)$.

Following \cite[Definition 9.2]{Bam20a}, the $P^*$-parabolic
neighborhood centered at $(x_0,t_0)\in M\times[0,T)$ is defined by
\[
P^*(x_0,t_0;A,-T^-,T^+)
:=
\left\{
(x,t)\in M\times[0,T)
\,\middle|\,
\begin{aligned}
&t\in[t_0-T^-,t_0+T^+],\\
&d_{W_1}^{g_{t_0-T^-}}
\bigl(\nu_{x_0,t_0;t_0-T^-},\nu_{x,t;t_0-T^-}\bigr)<A
\end{aligned}
\right\}.
\]
Here, $d_{W_1}^{g_t}$ denotes the $1$-Wasserstein distance between
probability measures on the metric space $(M,d_{g_t})$; see
\cite[Section 2]{Bam23}.

Fix $x_0\in M$ and let $T_i\nearrow T$. After passing to a
subsequence, $K(x_0,T_i;\cdot,\cdot)$ converges in
$C_{\mathrm{loc}}^\infty(M\times(0,T))$ to a positive solution of the
conjugate heat equation \cite[Lemma 2.2]{MM15},
\[
K(x_0,T;\cdot,\cdot)
:=
\lim_{i\to\infty}K(x_0,T_i;\cdot,\cdot)
\colon M\times[0,T)\longrightarrow\mathbb{R}.
\]
For $t\in[0,T)$, define $d\nu_{x_0,T;t}=K(x_0,T;\cdot,t)\,dg_t$. Then, for any $\varepsilon>0$,
\begin{align*}
    \lim_{i\to\infty}
\sup_{t\in[0,T-\varepsilon]}
d_{W_1}^{g_t}\bigl(\nu_{x_0,T_i;t},\nu_{x_0,T;t}\bigr)=0
\end{align*}
by \cite[Lemma 2.35]{Bam20b}. This $W_1$-convergence also yields
\[
\int_M\int_M d_t^2(x,y)\,
d\nu_{x_0,T;t}(x)\,d\nu_{x_0,T;t}(y)
\leq H_{n}(T-t),
\]
where $H_n:=\frac{(n-1)\pi^2}{2}+4$. Consequently, for each $t\in[0,T)$, there is a point $z\in M$
satisfying $\int_M d_t^2(z,y)\,d\nu_{x_0,T;t}(y)
\leq H_{n}(T-t)$; see \cite[Section 3]{Bam20a}. Such a point $(z,t)$ is called an
$H_{n}$-center of $(x_0,T)$.

We now recall the notion of pointed Nash entropy and some related results.
\begin{definition}\cite[Section 2.6]{Bam20b}
    $K(x_0,T;\cdot,\cdot)$ and $\nu_{x_0,T;t}$ are called conjugate heat kernels based at $(x_0,T)$. Define $f \in C^\infty(M \times [0,t))$ by $K(x_0,T;\cdot,t) = (2\pi\tau)^{-n/2} e^{-f} dg_t$, where $\tau := T-t$. Then the pointed Nash entropy at $(x_0,T)$ is given by
\[
\mathcal N_{x_0,T}(\tau) := \int_M f \, d\nu_{x_0,T;t} - \frac{n}{2}.
\]
\end{definition}
By \cite[Corollary 5.11]{Bam20a}, we know $\mathcal N_{x_0,T}(\tau)$ is independent of the sequence $T_i\nearrow T$, so $\mathcal N_{x_0,T}(0) := \lim_{\tau \searrow 0} \mathcal N_{x_0,T}(\tau) \in (-\infty,0]$ depends only on $x_0\in M$.
\subsubsection{Compactness and partial regularity of Ricci flow}
We first introduce the notion of orbifold Kähler-Ricci shrinker, which appears as a limit in Bamler's compactness theory for Kähler-Ricci flow.

\begin{definition}\cite[Definition 2.5]{CHM25}
    An orbifold Kähler-Ricci shrinker $(X,g,J,f)$ is a Kähler orbifold $(X,g)$ with complex structure $J$, together with a smooth real--valued function $f\in C^\infty(X)$ satisfying
\[
\operatorname{Ric}_g + \nabla^2 f = g,\qquad \mathcal{L}_{\nabla f}J = 0.
\]
We say $X$ has \textit{isolated singularities} if $X\setminus X_{\mathrm{reg}}$ is discrete. The isotropy group of an isolated singularity is a non--trivial finite subgroup of $U(n)$ acting freely on $\mathbb{C}^n\setminus\{0\}$.
\end{definition}

Because $\nabla f$ vanishes at each point of $X\setminus X_{reg}$, we can define a 1-parameter family of biholomorphisms of $(X,J)$ that preserves $X_{reg}$ by defining $\varphi_{-1}=id_X$ and $\partial_t \varphi_t(x)=\nabla f(\varphi_t(x))/|t|$. Set $g_t:=|t|\varphi_t^*g$, $(X,(g_t)_{t\in(-\infty,0)},J)$ defines an orbifold Kähler-Ricci flow. At each orbifold singular point, $X$ is locally biholomorphic to $\mathbb{C}^n/\Gamma$ for some finite subgroup $\Gamma\subseteq U(n)$. When all $\Gamma$ are trivial, we call $X$ a smooth Kähler‑Ricci shrinker.

Let $(M_i,(g_{i,t})_{t\in(-T_i,0]},(x_i,0))$ be a sequence of pointed Ricci flows on compact manifolds of the same dimension $n$ and $T_\infty:=\lim_{i\to\infty} T_i\in(0,\infty]$. By the results of \cite{Bam23} we may pass to a subsequence and obtain $\mathbb F$-convergence on compact time-intervals
\begin{align}\label{krs}
  \big(M_i,(g_{i,t})_{t\in(-T_i,0]},(\nu_{x_i,0;t})_{t\in(-T_i,0]}\big)
\xrightarrow[\,i\to\infty\,]{ \mathbb F,\mathfrak C }
\big(\mathcal X,(\nu_{x_\infty;t})_{t\in(-T_\infty,0]}\big),
\end{align}
within some correspondence $\mathfrak C$ (for more details see \cite{Bam23}). For some uniform $0<\tau_0<T_\infty$, $Y_0<\infty$ and any $i$, under the following non-collapsing assumption 
\begin{align}\label{noncollapse}
    \mathcal N_{x_i,0}(\tau_0)\ge -Y_0,
\end{align}
 the limit metric flow pair admits a regular-singular decomposition\cite[Theorem 2.4,2.5]{Bam20b}. If it also satisfy 
 \begin{align}
    \lim_{i\to\infty}\mathcal N_{x_i,0}(\tau)=W, \quad \text{for any} \quad \tau\in(0,T)
 \end{align}
 for some constant $W$, then the limit $\big(\mathcal X,(\nu_{x_\infty;t})_{t\in(-T_\infty,0]}\big)$ is a metric soliton (singular Ricci shrinker) \cite[Theorem 2.18]{Bam20b}. In compact 2-dim Kähler-Ricci flow, such metric soliton is an orbifold Kähler-Ricci shrinker $(X,g,J,f)$. More precisely, there exists an precompact exhaustion $(U_i)$ of $X_{reg}$ along with open embeddings $\psi_i: U_i \to M_i$ such that 
 \begin{align*}
    \psi_i^*g_{i,t}\to g_t, \qquad \psi_i^*J_M \to J, \qquad \psi_i^* \nu_{x_i,0;t}\to \nu_{x_\infty;t}.
 \end{align*}
in $C^{\infty}_{loc}(X_{reg}\times (-\infty,0))$ as $i\to \infty$. See also \cite[Theorem 2.37]{Bam20b}, \cite[Theorem 2.5]{HJ23}, \cite[Theorem 2.8]{CHM25} for more discussion.

\subsection{Polarized Fano fibration}
Here we recall the polarized Fano fibration
framework of \cite{SZ,HZ26}.

A \textit{fibration} is a surjective projective morphism $\pi \colon X \to Y $
between normal varieties such that $\pi_*\mathcal{O}_X = \mathcal{O}_Y$. 
\begin{definition}\cite{CT}
    A \textit{polarized affine cone} $(Y,\xi,T)$ consists of a normal affine variety $Y = \operatorname{Spec} R$ with a compact torus $T$-action admitting a unique fixed point, together with a vector $\xi \in \operatorname{Lie}(T)$ lying in the Reeb cone. Here the latter condition means that for the weight decomposition
\[
R = \bigoplus_{\alpha \in \operatorname{Lie}(T)^*} R_\alpha,
\]
one has $\langle \alpha, \xi \rangle > 0$ whenever $R_\alpha \neq 0$ and $\alpha \neq 0$.
\end{definition}
\begin{definition}[Polarized Fano fibration]
    A polarized Fano fibration $(\pi\colon X \to Y, \xi)$ is a fibration $\pi\colon X \to Y$ with the following properties:
\begin{enumerate}
    \item $\pi\colon X \to Y$ is a Fano fibration. That is, $X$ and $Y$ are normal varieties, $X$ is klt, and that $-K_X$ is $\pi$-ample and $\mathbb{Q}$-Cartier.
    \item $X$ and $Y$ are equipped with a $\pi$-equivariant torus action $T$, and $\xi \in \mathfrak{t} = \operatorname{Lie}(T)$.
    \item $(Y, T, \xi)$ is a polarized affine cone.
\end{enumerate}
\end{definition}
Hallgren and Zhang connect Bamler's compactness theory with algebraic geometry.
\begin{theorem}\cite{HZ26}
    \begin{enumerate}
      \item The singular Kähler-Ricci shrinker arised from (\ref{krs}) admits polarized Fano fibration structure;
      \item The metric singular set of X coincides with its complex analytic singular set.
    \end{enumerate}
\end{theorem}

\section{The smoothness of tangent shrinker}
\subsection{Tangent flow and $\mathbb{P}^1$-fibration}

In this section, we consider the Kähler tangent flow of $(M,g(t)_{t\in[0,T)},J_M)$ where $M$ admits a $\mathbb{P}^1$ fibration $p: M\to \Sigma$. Fix a conjugate heat kernel $(\nu_{x_0,T;t})$ based at $(x_0,T)$ with $x_0 \in p^{-1}(\Delta)$. For any sequence $\tau_i>0,\tau_i\searrow0$, we set $g_{i,t}:=\tau_i^{-1}g(T+\tau_it)$, by \cite[Theorem 2.8]{CHM25} we have the following $\mathbb{F}$-convergence
\begin{align*}
    \big(M_i,(g_{i,t})_{t\in[-\tau_i^{-1}T,0)},(\nu_{x_0,T;T+\tau_it})_{t\in[-\tau_i^{-1}T,0)}\big)
\xrightarrow[\,i\to\infty\,]{ \mathbb F,\mathfrak C }
\big(\mathcal X,(\nu_{x_\infty;t})_{t\in(-\infty,0)}\big),
\end{align*}
where $\big(\mathcal X,(\nu_{x_\infty;t})_{t\in(-\infty,0)}\big)$ is an orbifold Kähler-Ricci shrinker $(X,g,J,f)$ with isolated orbifold singularities and we set $X=X_{reg}\cup X_{sin}$.

\begin{proposition}\label{pr-bounded}
    For tangent shrinker $X$, there exists compact set $K\supset X_{\sin}$ with uniform constant $\Lambda<\infty$ and $i_0>0$ such that
    \begin{align*}
        \sup_{X}|\mathrm{Rm}|_g<\infty,\qquad
\operatorname{inj}_g(x)\ge i_0
\qquad\text{for any }x\in X\setminus K.
    \end{align*}
\end{proposition}
\begin{proof}
    By \cite{HZ26}, the metric and analytic singularities agree. Finite type and normality in dimension two imply that $X_{\sin}$ is finite. Then by \cite[Theorem 4.1]{CHM25}, $X$ has bounded curvature. \cite[Theorem 6.1]{Bam20a} provides a uniform lower bound on volume, and the Cheeger–Gromov–Taylor injectivity estimates imply the desired results.
\end{proof}

\begin{lemma}\label{lem:curves}
Every $x\in X_{\reg}$ lies on a compact irreducible holomorphic curve.
\end{lemma}
\begin{proof}
Fix $x\in X_{\mathrm{reg}}$. Since regular-fiber points are dense,
we may choose $x_i\to x$ such that $y_i=\psi_i(x_i)$ lies on a
regular fiber $G_i=p^{-1}(p(y_i))$. By \eqref{eq:local-class}, we know $\textstyle \int_{G_i}\omega_i=4\pi$. We write
$\widehat g_i=\psi_i^*g_i$ and
$\widehat\omega_i=\psi_i^*\omega_i$, and consider $T_i=\bigl[\psi_i^{-1}(G_i\cap\psi_i(U_i))\bigr]$ which satisfies
\begin{align*}
 \partial T_i=0\quad \text{in }U_i,
\qquad
\mathbf M_{\widehat g_i}(T_i)\le4\pi.
\end{align*}
Local compactness for integral currents $T_i$ and a
diagonal argument give a locally integral cycle $Z$ on
$X_{\mathrm{reg}}$ \cite[4.2.17]{Federer}. For every nonnegative $\chi\in C_c^\infty(X_{\mathrm{reg}})$,
smooth convergence gives
\[
  \mathbf M_g(\chi Z)
  \le \liminf_{i\to\infty}
      \mathbf M_{\widehat g_i}(\chi T_i)
  = \lim_{i\to\infty}T_i(\chi\widehat\omega_i)
  = Z(\chi\omega)
  \le \mathbf M_g(\chi Z).
\]
Thus $Z$ is calibrated by $\omega$. By King's theorem \cite{King} or \cite[Corollary 3.10.2]{TehYang}, $Z$ is
a positive holomorphic 1-cycle. The local mass bounds and smooth convergence of
$\widehat\omega_i$ imply
$T_i(\chi\widehat\omega_i)\to Z(\chi\omega)$, exhaustion by cutoffs gives
\[
  \mathbf M_g(Z)=\int_Z\omega\le4\pi.
\]

We next show that $x\in\operatorname{supp}Z$. For any sufficiently
small $r>0$, choose $\chi\in C_c^\infty(B_g(x,2r))$ with
$0\le\chi\le1$ and $\chi=1$ on $B_g(x,r)$. Since $x_i\to x$,
smooth convergence and the monotonicity formula give $T_i(\chi\widehat\omega_i)\ge cr^2$ for sufficiently large $i$. Passing to the limit yields
$Z(\chi\omega)\ge cr^2$. Hence $x\in\operatorname{supp}Z$.

Applying the Remmert--Stein theorem \cite[Chapter II, 8.7]{De} in quotient charts, the
closure of $\operatorname{supp}Z$ across the finite set
$X_{sin}$ is an analytic curve in $X$. Let $D$ be
an irreducible component containing $x$, we have
\[
  \operatorname{Area}_g(D)\le\mathbf M_g(Z)\le4\pi.
\]

By Proposition~\ref{pr-bounded} and the monotonicity formula \cite[Chapter~8, \S3]{simon2014introduction}, there
are a compact set $K\subset X$ and constants $r_0,c_0>0$ such that
\[
  \operatorname{Area}_g(D\cap B_g(z,r_0))\ge c_0
  \qquad\text{for all }z\in D\setminus K.
\]
If $D$ were noncompact,
its closedness would imply that it is unbounded. We could then
choose infinitely many points $z_j\in D\setminus K$ whose balls
$B_g(z_j,r_0)$ are pairwise disjoint. The preceding estimate
would give $\operatorname{Area}_g(D)=\infty$, a contradiction.
Thus $D$ is compact.
\end{proof}
Combining the above lemma with the structure of polarized Fano fibrations, we obtain the following explicit fibration structure.
\begin{proposition}
    $Y\simeq \mathbb{A}^1$, the general fiber is $\PP^1$,
and $X_{\sin}\subseteq\pi^{-1}(o)$
\end{proposition}
\begin{proof}
If $\dim_\C Y=2$, $\pi$ is birational by normality and connected
fibers. A compact curve through a regular point in its isomorphism
locus contradicts affineness of $Y$. 

If $\dim_\C Y=0$, $X$ is
projective, a general smooth ample curve avoids its finite
singular set and has positive self-intersection. Its compact
neighborhood embeds into $V$, contradicting $Q_V\le0$.
Thus $\dim_\C Y=1$.

A normal affine curve with the positive Reeb
action is $\A^1$\cite[Proposition 2.15]{XZ26}. Relative anticanonical ampleness
and adjunction give genus zero for the smooth general fiber \cite[Section 4.1.1]{XZ26}. Because $X_{\sin}$ is the analytic singular set and $T$ biholomorphiclly acts on $X$, so $T\times X_{\sin} \subseteq X_{\sin}$, then $X_{\sin} \subseteq \pi^{-1}(o)$.   
\end{proof}

\subsection{The local topology near the singular fiber}

Take $o\in X_{\sin}$ with a chart $B^4/\Gamma$.
Choose a $\Gamma$-invariant K\"ahler potential, normalized such that $\omega=dd^c\varphi$ for $\varphi(z)=H(z,\bar z)+O(|z|^3)$ where $H$ is positive definite and
$d^c={\sqrt{-1}}(\bar\partial-\partial)/2$.
For small $r>0$, set $\Omega_r=\{\varphi<r^2\}/\Gamma$, $L_r=\partial\Omega_r\simeq S^3/\Gamma$, and $\alpha=d^c\varphi$. With the outward orientation, Stokes' theorem gives
\begin{equation}\label{eq:action}
  B_r:=\int_{L_r}\alpha\wedge d\alpha
      =\int_{\Omega_r}\omega^2>0,
  \qquad B_r=O(r^4).
\end{equation}

Fix $r$. Since $H^2(L_r;\mathbb R)=0$, in the collar direction give one-forms $\gamma_i$
on a fixed collar of $L_r$ satisfying $d\gamma_i=\psi_i^*\omega_i-\omega$ and $\gamma_i\longrightarrow0$ smoothly on a smaller fixed collar.
Set $L_{r,i}=\psi_i(L_r)$ and
$\alpha_i=(\psi_i^{-1})^*(\alpha+\gamma_i)$.
Then $d\alpha_i=\omega_i$, and orient $L_i$ so that $\psi_i:L\to L_i$ preserves orientation.
\begin{equation}\label{eq:action-limit}
  B_{r,i}:=\int_{L_{r,i}}\alpha_i\wedge d\alpha_i
  \longrightarrow B_r.
\end{equation}
In fact, the link $L_{r,i}\subset V$ bounds a unique relatively compact side and let $W_i$ denote its closure and $E_i$ denote $X \setminus W_i$, a compact manifold with boundary. More precisely, we have the following lemma.
\begin{lemma}\label{lem:W-topology}
In the complex orientation,
\[
H_2(W_i;\mathbb Q)\hookrightarrow H_2(V;\mathbb Q),
\qquad Q_{W_i}<0,\qquad b_2(W_i)\le1.
\]
Moreover, $H_1(W_i;\mathbb Q)=H_3(W_i;\mathbb Q)=0$.
\end{lemma}

\begin{proof}
The link separates $V$ because $V$ is simply connected
\cite[Theorem 4.4.6]{HR}. One-endedness implies that exactly
one side is relatively compact. Let $W_i$ be its closure
and set $E_i=V\setminus\operatorname{int}W_i$. Since $L_{r,i}$ is a rational homology sphere,
Mayer--Vietoris sequence gives
\[
H_2(V;\mathbb Q)\cong H_2(W_i;\mathbb Q)\oplus H_2(E_i;\mathbb Q),
\qquad H_1(W_i;\mathbb Q)=0.
\]
The inclusion preserves intersection pairings.
Poincar\'e--Lefschetz duality makes $Q_{W_i}$ nondegenerate.
By \eqref{eq:V}, $Q_V$ is negative semidefinite with negative
index at most one. Thus $Q_{W_i}$ is negative definite
and $b_2(W_i)\le1$. Since $W_i$ and its boundary are connected, the exact
sequence of the pair and $H^1(W_i;\mathbb Q)=0$ give
$H^1(W_i,\partial W_i;\mathbb Q)=0$. Duality yields
\[
H_3(W_i;\mathbb Q)\cong H^1(W_i,\partial W_i;\mathbb Q)=0.
\]
\end{proof}

\begin{lemma}\label{lem:exact}
    For sufficiently small $r$ and large $i$, $c_1(W_i;\R)=0$ and $[\omega_i|_{W_i}]=0$.
\end{lemma}
\begin{proof}
$W_i$ is oriented by the complex orientation of $M$. Set $\kappa_i=c_1(W_i)\in H^2(W_i;\mathbb Z)$ and $h=|H_1(L_{r,i};\mathbb Z)|$. The integer $h$ depends only on the link type. Choose $r$
small and $i$ large so that $0<B_{r,i}<{(2\pi)^2}/{h}$. The exact sequence of the pair $(W_i,\partial W_i)$ gives
\[
H^2(W_i,\partial W_i;\mathbb R)
\cong H^2(W_i;\mathbb R).
\]
Let $\bar\kappa_i$ be the relative lift of $\kappa_i$ over
$\mathbb R$, and define $\kappa_i^2
=\langle\bar\kappa_i\smile\kappa_i,
[W_i,\partial W_i]\rangle$. Since $H^2(\partial W_i;\mathbb Z)$ has order $h$,
the class $h\kappa_i$ has an integral relative lift.
Thus $h\kappa_i^2\in\mathbb Z$. By
Lemma~\ref{lem:W-topology}, $\kappa_i^2\le0$ and $\kappa_i^2\le-1/h$.

Choose a closed two-form $\theta_i$ supported in the interior
of $W_i$ and representing $\bar\kappa_i$.
Equation~\eqref{eq:local-class} gives $\omega_i=2\pi\theta_i+d\beta_i$. Near the boundary, $d\beta_i=d\alpha_i$, Stokes' theorem and the interior support of $\theta_i$ yield
\begin{equation}\label{eq:relative-volume}
\int_{W_i}\omega_i^2
=(2\pi)^2\kappa_i^2
+\int_{\partial W_i}\alpha_i\wedge d\alpha_i.
\end{equation}

If the boundary orientations were opposite, then
\[
0<\int_{W_i}\omega_i^2
=(2\pi)^2\kappa_i^2-B_{r,i}<0,
\]
a contradiction. Thus the boundary orientations agree.
Since $\psi_i$ preserves the ambient orientation, it maps
the outer collar into $V\setminus W_i$. If $\kappa_i$ were nonzero over $\mathbb R$, then
\[
0<\int_{W_i}\omega_i^2
\le -\frac{(2\pi)^2}{h}+B_{r,i}<0.
\]
Hence $c_1(W_i;\mathbb R)=0$, and \eqref{eq:local-class}
gives $[\omega_i|_{W_i}]=0$.
\end{proof}
Let $N$ be the compact minimal-resolution neighborhood of one
of the small quotient balls. Its exceptional divisor is a tree
of smooth rational curves $E_1,\ldots,E_\ell$ with negative-definite
intersection matrix. Write
\[
 E_\alpha^2=-b_\alpha,\qquad b_\alpha\ge2,\qquad
 K_N=\rho^*K_X+A,\qquad A=\sum_{\alpha=1}^{\ell}a_\alpha E_\alpha.
\]
The lower bounds $a_\alpha>-1$ follow from the klt condition. Since $|\Gamma|<\infty$, we can select a $m\in \mathbb{Z}_{>0}$ s.t. $mK_X\sim0$ in a small neighborhood, thus $c_1(K_N)=[A] \in H^2(N;\Q)$. Adjunction gives
\begin{equation}\label{eq:discrepancy-square}
 A\cdot E_\alpha=b_\alpha-2,\qquad
 A^2=\sum_{\alpha=1}^{\ell}a_\alpha(b_\alpha-2).
\end{equation}
\begin{lemma}\label{lem:replacement}
    For a fixed sufficiently small link and sufficiently large $i$, then
\begin{equation}\label{eq:replacement}
    b_2(W_i)=\ell+A^2.
\end{equation}
\end{lemma}
\begin{proof}
    Resolve an open neighborhood of the closed quotient ball.
The resolution identifies a neighborhood of $\partial N$
biholomorphically with a neighborhood of the link. On $U$,
$\psi_i^*J\to J_X$ smoothly. For large $i$, we can deform $J$
through almost-complex structures, keeping it fixed outside
$\psi_i(U)$, to obtain $J_i'$ satisfying $\psi_i^*J_i'=J_X$ near the link. Thus $J_i'$ is homotopic to $J$ on $M$ and agrees with the
complex structure of the resolution along the gluing collar.

Replacing $W_i$ by $N$ therefore produces a closed oriented
almost-complex manifold $\widehat M_i=(M\setminus\operatorname{int}W_i)\cup N$. Set $E_i=M\setminus\operatorname{int}W_i$.
The new structure on this common exterior is homotopic to
the original one, so its first Chern class is unchanged.
Since the common boundary is a rational homology sphere,
Mayer--Vietoris sequence gives orthogonal decompositions
\[
 H^2(M;\Q)=H^2(E_i;\Q)\oplus H^2(W_i;\Q),\qquad
 H^2(\widehat M_i;\Q)=H^2(E_i;\Q)\oplus H^2(N;\Q).
\]
The Chern class on $W_i$ vanishes over $\R$ by
Lemma~\ref{lem:exact}, and $c_1(TN)=-[A]$. Consequently
\[
 c_1(\widehat M_i)^2-c_1(M)^2=A^2.
\]

Lemma~\ref{lem:W-topology} gives
$\chi(W_i)=1+r_i$ and $\sigma(W_i)=-r_i$.
The resolution neighborhood retracts onto its exceptional tree,
so $\chi(N)=1+\ell$ and $\sigma(N)=-\ell$. So
\[
 \chi(\widehat M_i)-\chi(M)=\ell-r_i,\qquad
 \sigma(\widehat M_i)-\sigma(M)=-\ell+r_i.
\]
For every closed almost-complex four-manifold,
$c_1^2=2\chi+3\sigma$.
Then we have 
$A^2=2(\ell-r_i)+3(-\ell+r_i)=r_i-\ell$.
\end{proof}

Let $o_1,\ldots,o_s$ be the singular points of $X$. Write
\[
F_0=\sum \nolimits_{j=1}^k m_jD_j,\qquad
Q=(D_j\cdot D_l)_{jl},\qquad
m=(m_1,\ldots,m_k)^T.
\]
Since $F_0$ is principal, $F_0\cdot D_j=0$ for every $j$,
so $Qm=0$. For any $v\in\mathbb Q^k$, this gives
\begin{equation}\label{eq:graph}
v^TQv
=-\sum_{j<l}(D_j\cdot D_l)m_jm_l
\left(\frac{v_j}{m_j}-\frac{v_l}{m_l}\right)^2.
\end{equation}
Distinct components have nonnegative intersection.
Since $F_0$ is connected, equality holds precisely when
all ratios $v_j/m_j$ are equal. Thus
\[
Q\le0,\qquad \ker_{\mathbb Q}Q=\mathbb Qm.
\]
\begin{proposition}\label{prop:k}
For sufficiently small fixed links around the singular points
and large $i$, the regions $W_{\nu,i}$ are
pairwise disjoint and
\begin{equation}\label{eq:combined-rank}
  k+\sum_{\nu=1}^{s}b_2(W_{\nu,i})\le b_2(V)\le2.
\end{equation}
\end{proposition}

\begin{proof}
Choose quotient balls $B_\nu\Subset B_\nu^+$ around $o_\nu$,
with pairwise disjoint closures $\overline{B_\nu^+}$.
Set $L_\nu=\partial B_\nu$ and $L_\nu^+=\partial B_\nu^+$,
and choose $L_\nu^+$ transverse to every $D_j$.
Replace each part of $D_j$ inside $B_\nu^+$ by a rational
two-chain on $L_\nu^+$ with the same oriented boundary.
Such chains exist because $H_1(L_\nu^+;\mathbb Q)=0$.
This gives rational two-cycles $\beta_j$
outside $\textstyle\bigcup_\nu\overline{B_\nu}$.

Let $K$ be a compact neighborhood of the curves and balls,
and let $\widetilde K$ be its preimage under a resolution.
Write $N_\nu$ for the preimage of $\overline{B_\nu^+}$
and $K^\circ$ for their common exterior, identified with
the corresponding subset of $X_{\mathrm{reg}}$.
Put $L^+=\textstyle \bigsqcup_\nu L_\nu^+$.
Since $H_1(L^+;\mathbb Q)=H_2(L^+;\mathbb Q)=0$,
Mayer--Vietoris sequence gives the orthogonal decomposition
\[
H_2(\widetilde K;\mathbb Q)
=
H_2(K^\circ;\mathbb Q)
\oplus\bigoplus_\nu H_2(N_\nu;\mathbb Q).
\]
The exact sequence of the pair gives $H_2(K^\circ;\mathbb Q)
\cong H_2(K^\circ,L^+;\mathbb Q)$, so $[\beta_j]$ is the unique lift of the truncated curve.
The numerical pullback class $P_j$ of $D_j$ therefore has
the form
\[
P_j=[\beta_j]+u_j,
\qquad
u_j\in\bigoplus_\nu H_2(N_\nu;\mathbb Q).
\]
It is orthogonal to every exceptional curve.
The exceptional intersection form is negative definite,
hence $u_j=0$ and
$D_j\cdot D_l=P_j\cdot P_l=\beta_j\cdot\beta_l$.

Write $\omega=d\alpha_\nu$ near $\overline{B_\nu^+}$.
The removed curve piece and its replacement have the same
oriented boundary, so Stokes' theorem gives equal integrals.
Thus
\begin{equation}\label{eq:capped-pairings}
\beta_j\cdot\beta_l=D_j\cdot D_l,
\qquad
\int_{\beta_j}\omega=\int_{D_j}\omega.
\end{equation}

Choose disjoint two-sided collars of $L_\nu$ and a point
$p_\nu$ in the outer part of each collar.
Join these points, the cycles $\beta_j$, and their intersection
perturbations by paths outside
$\textstyle\bigcup_\nu\overline{B_\nu}$.
Let $C$ be the resulting compact connected set.
For large $i$, the map $\psi_i$ is defined near $C$ and
all the collars, with image in $V$.

By Lemma~\ref{lem:exact}, $\psi_i(p_\nu)$ lies outside
$W_{\nu,i}$. Since $\psi_i(C)$ is connected and avoids
$L_{\nu,i}$, it lies outside $W_{\nu,i}$ for every $\nu$.
If $W_{\mu,i}\subset\operatorname{int}W_{\nu,i}$,
the collar of $L_{\mu,i}$ would lie inside $W_{\nu,i}$,
since it is connected and avoids $L_{\nu,i}$.
This contradicts $\psi_i(p_\mu)\in\psi_i(C)$.
The compact sides of disjoint separating links are disjoint
or nested. Hence the $W_{\nu,i}$ are pairwise disjoint,
and every $\psi_i(\beta_j)$ lies in their common exterior.

Define
\[
\Phi_i:\mathbb Q^k\longrightarrow H_2(V;\mathbb Q),
\qquad
\Phi_i(e_j)=(\psi_i)_*[\beta_j].
\]
Their intersection matrix is $Q$, so $\ker\Phi_i\subset\ker Q=\mathbb Qm$ and
\[
\int_{\Phi_i(m)}\omega_i
=
\sum_jm_j\int_{\beta_j}\psi_i^*\omega_i
\longrightarrow
\int_{F_0}\omega>0.
\]
Thus $\Phi_i(m)\ne0$ for large $i$, and $\Phi_i$ is injective.

Lemma~\ref{lem:W-topology} identifies each
$H_2(W_{\nu,i};\mathbb Q)$ with a negative-definite subspace
of $H_2(V;\mathbb Q)$. These subspaces are mutually orthogonal
and form a direct sum $U_i$. The space
$A_i=\operatorname{im}\Phi_i$ is orthogonal to $U_i$,
because its cycles lie outside the regions $W_{\nu,i}$.
Negative definiteness gives $A_i\cap U_i=0$. Therefore
\[
k+\sum_{\nu=1}^s b_2(W_{\nu,i})
=\dim A_i+\dim U_i
\le b_2(V)\le2.
\]
\end{proof}

\subsection{Removable singularities and smooth limit}
Combining the $\PP^1$ fibration structure and the local topology near the singular fiber, we can prove the orbifold singularities of $X$ are removable.
\begin{proposition}
    The tangent flow $X$ is smooth shrinking cylinder $\mathbb{P}^1\times \mathbb{C}$ or BCCD shrinker $\operatorname{Bl}_p(\mathbb{P}^1\times \mathbb{C})$.
\end{proposition}
\begin{proof}
    Suppose $X$ has $s\ge1$ singularities which are contained in $\pi^{-1}(o)$. Fix the small neighborhoods and
one common sufficiently large convergence index from
Proposition~\ref{prop:k}. Write
$r_\nu=b_2(W_\nu)$, and set
\[
 L=\sum_{\nu=1}^{s}\ell_\nu>0,\qquad
 R=\sum_{\nu=1}^{s}r_\nu,\qquad
 B=\sum_{\nu,\alpha}b_{\nu\alpha},\qquad
 \varepsilon=\sum_{\nu,\alpha}
       (1+a_{\nu\alpha})(b_{\nu\alpha}-2)\ge0.
\]
The nonnegativity uses $a_{\nu\alpha}>-1$ and
$b_{\nu\alpha}\ge2$. Summing \eqref{eq:replacement} and using
\eqref{eq:discrepancy-square} gives
\begin{equation}\label{eq:weight-sum}
 B=3L-R+\varepsilon.
\end{equation}

Let $\rho:\widetilde X\to X$ be the minimal resolution.
The map $\pi\circ\rho$ is proper with connected fibers and
general fiber $\PP^1$, the central reduced support has $k+L\ge2$ components.
Write $\widetilde D_j^2=-c_j$ for the strict transforms of the
original $k$ components. Each component of a reducible fiber
has negative square, hence
$c_j\ge1$ and $C:=\textstyle\sum\nolimits_j c_j\ge k$.
The trace of the reduced-support intersection matrix is $-C-B$.

By lemma \ref{bea}, the same fiber is obtained from a smooth fiber $\PP^1$ by
$k+L-1$ point blowups. If $u$ centers are smooth support points
and $v$ are nodes, then
\[
 u+v=k+L-1,\qquad
 -C-B=-2u-3v=-3(k+L-1)+u.
\]
The first blowup is at a smooth point, so $u\ge1$.
Substitution of \eqref{eq:weight-sum} and
\eqref{eq:combined-rank} yields
\begin{equation}\label{eq:trace-obstruction}
 1\le u
 =3k-3-C+R-\varepsilon
 \le2k-3+R-\varepsilon
 \le k-1-\varepsilon
 \le1-\varepsilon.
\end{equation}
Thus 
$k=2$, $R=0$, and $\varepsilon=0$.
Since every $1+a_{\nu\alpha}$ is strictly positive,
$\varepsilon=0$ forces every $b_{\nu\alpha}=2$.
Equation~\eqref{eq:discrepancy-square} then gives $A_\nu^2=0$,
so \eqref{eq:replacement} gives $r_\nu=\ell_\nu$.
This contradicts $0=R=L>0$.
Therefore $X$ is smooth. Because $X$ is a smooth $\PP^1$-fibration over $\mathbb{C}$, the classification \cite[Theorem 1.1]{LW26} implies $X$ is shrinking cylinder or BCCD shrinker.
\end{proof}
To exclude the shrinking cylinder, we first show that along a
sequence of scales, the $H_4$-centers lie a bounded distance
from singular fiber $F_q$ in the rescaled metrics.
Fix $x_0\in F_q$ and conjugate heat kernel
$K(x_0,T;\cdot,\cdot)$ obtained along a sequence $t_j\nearrow T$.
For each $\tau\in(0,T)$, let $(z_\tau,T-\tau)$ be an
$H_4$-center of $(x_0,T)$.
\begin{lemma}\label{lem:visible-subsequence}
There exist $A<\infty$ and $\tau_i\searrow0$ such that
\[
 d_{g(T-\tau_i)}(z_{\tau_i},F_q)\le A\sqrt{\tau_i}.
\]
\end{lemma}
\begin{proof}
Let $s_q$ be the canonical section of $\mathcal O_\Sigma(q)$,
choose a smooth Hermitian metric $h$ and set $u=p^*|s_q|_h^2$, the Poincaré-Lelong formula implies
\begin{align}\label{ple}
 \sqrt{-1}\partial \bar{\partial} \operatorname{log} u=2\pi[F_q]-p^*\Theta_h.
\end{align}
Since \(\Sigma\) is a compact Riemann surface, every \(2\)-form on \(\Sigma\) can be expressed as \(f\omega_\Sigma\) for some smooth function \(f\), the Schwarz estimate $p^*\omega_{\Sigma}\leq C\omega(t)$ implies 
\begin{align*}
  |\Delta_{\omega(t)}u|+
|\operatorname{tr}_{\omega(t)}p^*\Theta_h|\le C,
\end{align*}
Fix $\tau>0$ and let $s=T-\tau$. Direct computation implies
\begin{align*}
  E_j(s)=:\int_M u\,d\nu_{x_0,t_j;s}, \qquad E_j'(s)=-\int_M\Delta_{\omega(s)}u\,d\nu_{x_0,t_j;s}, \qquad E_j(t_j)=u(x_0)=0,
\end{align*}
Integrating we have \(0\le E_j(s)\le C(t_j-s)\), take limit we obtain $\textstyle \int_M u(y)\,d\nu_{x_0,T;T-\tau}(y)\le C\tau$. The Jesen inequality implies
\begin{align}\label{log}
  L(\tau):=\int_M\log u\,d\nu_{x_0,T;T-\tau}(y)
\le \log\!\left(\int_Mu\,d\nu_{x_0,T;T-\tau}(y)\right)
\le\log\tau+C.
\end{align}
By the direct computation and (\ref{ple})
\begin{align*}
L'(\tau)
&=\int_M
\operatorname{log}u \,
\Delta_{\omega(T-\tau)}
K(x_0,T;y,T-\tau)\,
dV_{g(T-\tau)}(y)\\
&=2\pi\int_F
K(x_0,T;y,T-\tau)\,\omega(T-\tau)-
\int_M
\operatorname{tr}_{\omega(T-\tau)}(p^*\Theta_h)\,
d\nu_{x_0,T;T-\tau}(y)\\
&\leq 2\pi\int_F
K(x_0,T;y,T-\tau)\,\omega(T-\tau)+C.
\end{align*}
We set $I(\tau)=\textstyle \int_F K(x_0,T;y,T-\tau)\,\omega(T-\tau)$. If $\limsup\nolimits_{\tau\to 0}\tau I(\tau)<1/(2\pi)$,
then $L'(\tau)\le b/\tau+C$ for some $b<1$ and all small $\tau$.
Integration gives $L(\tau)\ge b\log\tau-C$, contradicting
\eqref{log}. Thus we can choose $\tau_i\searrow0$
with $\tau_iI(\tau_i)\ge1/(4\pi)$. 
Applying \cite[Theorem 7.2]{Bam20a} to the approximating
kernels based at $t_j<T$, using the uniform entropy lower
bound, and passing to the limit gives
\begin{equation}\label{eq:visibility-gaussian}
 K(x_0,T;y,T-\tau)\le C\tau^{-2}
 \exp\!\left(-\frac{d_{g_{(T-\tau)}}(z_\tau,y)^2}{C\tau}\right).
\end{equation}
Two $H_4$ centers are
at distance at most $2\sqrt{H\tau}$, so the estimate holds
for any choice of center, after changing $C$.
Since $\textstyle\int_F\omega(T-\tau)=4\pi\tau$, integration implies
\[
 \tau I(\tau)\le C
 \exp\!\left(-\frac{d_{g_{(T-\tau)}}(z_\tau,F)^2}{C\tau}\right).
\]
The lower bound along $\tau_i$ proves the assertion.
\end{proof}
Now we exclude the shrinking cylinder as the tangent flow.
\begin{theorem}\label{prop:visibility-bccd}
The tangent flow $X$ is the BCCD shrinker $\operatorname{Bl}_p(\mathbb{P}^1\times \mathbb{C})$.
\end{theorem}
\begin{proof}
We first consider a special tangent flow, we choose $\tau_i \searrow 0$ which satisfies lemma \ref{lem:visible-subsequence} and set $g_i(t)=\tau_i^{-1}g(T+\tau_i t)$. After passing to a subsequence, let $(X,g_X(t),J_X)$ be the
corresponding tangent shrinker which is shrinking cylinder or BCCD shrinker.
Let $\psi_i$ be the regular convergence maps at time $-1$,
and let $(z_i,T-\tau_i)$ be the $H_4$ center of $(x_0,T)$.
Lemma \ref{lem:visible-subsequence} implies $d_{g_i(-1)}(z_i,F)\le A$. 

Suppose that $X$ is the shrinking cylinder.
Fix $o\in X$. Smooth convergence of the rescaled kernel gives
\[
 \liminf_{i\to\infty}
 \tau_i^2 K(x_0,T;\psi_i(o),T-\tau_i)>0.
\]
and \cite[Theorem 7.2]{Bam20a} implies 
\[
 \tau_i^2 K(x_0,T;\psi_i(o),T-\tau_i)
 \le C\exp\!\left(
 -\frac{d_{g_i(-1)}(z_i,\psi_i(o))^2}{C}
 \right).
\]
Hence $d_{g_i(-1)}(z_i,\psi_i(o))\le B$ for some $B<\infty$.
Choose $w_i\in F$ with $d_{g_i(-1)}(z_i,w_i)\le A$.
Then $d_{g_i(-1)}(\psi_i(o),w_i)\le A+B$. By the smooth converge, we know $v_i:=\psi_i^{-1}(w_i)
 \in\overline B_{g_X(-1)}(o,2(A+B))$. After passing to a subsequence,
$v_i\to v_\infty$.

Replace $\psi_i$ by $\psi_i\circ\chi_i$, where
$\chi_i(v_\infty)=v_i$ and $\chi_i\to\mathrm{id}$ smoothly.
Then $\psi_i(v_\infty)=w_i$, and smooth convergence is unchanged.
Let $C$ be the $\mathbb P^1$-fiber through $v_\infty$. It has self-intersection zero.
Apply 
\cite[Corollary 2.3]{CCD} on a fixed neighborhood of $C$, it gives embedded
$J$-holomorphic spheres $D_i\subset M$ such that $w_i\in D_i$ and $D_i^2=0$, so
\[
 \tau_i^{-1}\int_{D_i}\omega(T-\tau_i)
 =
 \int_C\psi_i^*\!\left(\tau_i^{-1}\omega(T-\tau_i)\right)
 \longrightarrow \int_C\omega_X(-1).
\]
In particular, $\textstyle\int_{D_i}\omega(T-\tau_i)=O(\tau_i)$. The Schwarz estimate now gives
\[
 0\le\int_{D_i}p^*\omega_\Sigma
 \le C\int_{D_i}\omega(T-\tau_i)
 \longrightarrow0.
\]
If $p|_{D_i}$ were nonconstant, it would have positive integer
degree onto $\Sigma$, so
\[
 \int_{D_i}p^*\omega_\Sigma
 =\deg(p|_{D_i})\int_\Sigma\omega_\Sigma >0.
\]
Thus $p|_{D_i}$ is constant for large $i$.
Since $w_i\in D_i\cap F_q$, this constant is $q$, so
$D_i\subset F_q$.
The irreducible curve $D_i$ must equal $C_1$ or $C_2$.
Both have self-intersection $-1$, contradicting $D_i^2=0$. So the tangent shrinker $X$ is the BCCD shrinker. Finally, entropy convergence gives $\lim\nolimits_{\tau\searrow0}\mathcal N_{x_0,T}(\tau)
 =\mathcal{N}_{\mathrm{BCCD}}$. Since $\mathcal{N}_{\mathrm{BCCD}}\neq\mathcal{N}_{\mathrm{cyl}}$, hence $X$ is the BCCD shrinker for all tangent flow.
\end{proof}
\section{The Type I estimate}
\begin{lemma}
\label{lem:typeI-exterior-link}
Let $L\cong S^3/\Gamma$ be the link of an isolated K\"ahler
orbifold point, and let $U$ be an annular neighborhood of $L$.
Suppose there are embeddings $\psi_j:U\to V\setminus F_q$
such that the pulled-back K\"ahler forms and complex structures
converge smoothly to those on $U$. Then $\Gamma$ is trivial.
\end{lemma}
\begin{proof}
    We use the same notation as in Section 3.2. Write $L\cong S^3/\Gamma$ for the link, set $L_j=\psi_j(L)$,
and let $W_j$ be the compact region bounded by $L_j$ in $V$.
By Lemma~\ref{lem:W-topology}, $Q_{W_j}$ is negative definite
and $b_1(W_j)=b_3(W_j)=0$.

Since $F_q$ is connected and disjoint from $L_j$, it lies
entirely inside or outside $W_j$. The first case is impossible:
$[F_q]\ne0$ in $H_2(V;\mathbb Q)$ and $F_q^2=0$, contrary
to the negative definiteness of $Q_{W_j}$.
Thus $W_j\subset V\setminus F_q$. Since
$Q_{V\setminus F_q}=0$, the negative definiteness of $Q_{W_j}$ forces $b_2(W_j)=0$. Hence $c_1(W_j;\mathbb R)=0$ and $[\omega(s_j)|_{W_j}]=0$.

The link $L_j$ embeds in
$V\setminus F_q\cong\mathbb P^1\times\Delta^*$, which embeds
smoothly in $\mathbb R^4$. By
\cite[Proposition 2.14(i)]{CHM25}, $\Gamma\subset SU(2)$, so the minimal resolution is a crepant resolution i.e. $K_N=\rho^*K_{B/\Gamma}$ ($A=0$). Let $\ell$ be the number of exceptional curves. The convergence of complex structures and the boundary
orientation allow the replacement in
Lemma~\ref{lem:replacement}. Its Chern-number calculation gives $b_2(W_j)=\ell+A^2$. Thus $\ell=0$, and $o$ is smooth.
\end{proof}

Using the topological lemma above and a contradiction argument, we extend the Type I estimate established in \cite{XZ} to the generalized form below. We let \(\mathcal{F}\) stand for the union of singular fibers of the $M$.
\begin{proposition}
\label{prop:typeI-exterior}
For any $\varepsilon>0$, there exists $C_\varepsilon<\infty$
such that
\[
|\operatorname{Rm}(g(t))|(x)\le \frac{C_\varepsilon}{T-t}
\]
for every $(x,t)\in M\times [0,T)$ satisfying $P^{*-}(x,t;\varepsilon\sqrt{T-t})
\cap(\mathcal F\times[0,T))=\varnothing$.

\end{proposition}
\begin{proof}
Suppose otherwise, take $t_i\nearrow T$, $\tau_i=T-t_i$, with
\begin{equation}\label{eq:typeI-exterior-bad}
 \tau_i|\operatorname{Rm}(g(t_i))|(x_i)\to\infty,\qquad
 P^{*-}(x_i,t_i;\varepsilon\sqrt{\tau_i})
 \cap(\mathcal F\times[0,T))=\varnothing .
\end{equation}
The regular-fiber estimate \cite[Corollary 1.2]{XZ} implies,
after passing to a subsequence, $p(x_i)\to q\in\Delta$.
Fix $V=p^{-1}(D)$ about $F_q$. Rescale by
$g_i(s)=\tau_i^{-1}g(t_i+\tau_i s)$, by Bamler's compactness theory, we have
\begin{align}\label{one}
      \big(M,(g_{i}(s))_{s\in(-\tau_i^{-1}t_i,0]},(\nu_{x_i,0;t})_{t\in(-\tau_i^{-1}t_i,0]}\big)
\xrightarrow[\,i\to\infty\,]{ \mathbb F,\mathfrak C }
\big(\mathcal X,(\nu_{x_\infty;t})_{s\in(-\infty,0]}\big),
\end{align}
where $(\mathcal X,(\nu_{x_\infty;t})_{s\in(-\infty,0]})$ is smooth away from away a codimension $4$ singular set $\mathcal{S}$ and we set $\mathcal{X}=\mathcal{R}\cup\mathcal{S}$. By the $\epsilon$-regularity theorem \cite[Theorem 10.2]{Bam20a} and (\ref{eq:typeI-exterior-bad}), we know the flow is singular at $x_{\infty}$, so the metric flow $\mathcal X$ has a nontrivial  tangent flow at $x_{\infty}$ \cite[Theorem 2.6, 2.11]{Bam20b}, more precisely, there is a sequence $\lambda_j\to 0$ such that 
\begin{align}\label{two}
    \lambda_j^{-1}
\bigl(\mathcal X,(\nu_{x_\infty;t})_{t\le0}\bigr)
\xrightarrow[\,j\to\infty\,]{ \mathbb F,\mathfrak C }
\bigl(\mathcal X',(\nu_{x'_\infty;t})_{t\le0}\bigr),
\end{align}
this limit is a metric soliton and we set its time slice at $-1$ time be the $(X',g',J',f)$ which is an orbifold Kähler-Ricci shrinker with isolated orbifold singularities. The smooth convergence in (\ref{two}) implies that the Cheeger-Gromov embedding maps $\varphi_j:U_j\to \mathcal{R}_{-\lambda_j^2}$ between the regular part of $X'$ and $\mathcal{X}$, where $\textstyle \cup_j U_j$ represent the exhaustion of $X'_{reg}$. (\ref{one}) and (\ref{two}) imply that we can choose a diagonal sequence $i(j)\to\infty$ for the 
regular convergence, and denote the corresponding embeddings by
$\psi_j:U_j\to M$. Set
\[
\delta_j=\lambda_j^2\tau_{i(j)},
\qquad
s_j=t_{i(j)}-\delta_j.
\]
Then
\[
\delta_j^{-1}\psi_j^*g(s_j)
=\lambda_j^{-2}\psi_j^*g_{i(j)}(-\lambda_j^2)
\longrightarrow g' \qquad \text{in} \,\,C^\infty_{\mathrm{loc}}(X'_{\mathrm{reg}}).
\]

Fix $K\Subset X'_{\mathrm{reg}}$ and let $z_j$ be a $H_4$-center
of $(x_{i(j)},t_{i(j)})$ at time $s_j$. The smooth convergence on $K$ gives lower bound of the heat kernel and \cite[Theorem 7.2]{Bam20a} implies
\begin{align*}
    d_{g(s_j)}(\psi_j(y),z_j)\le C_K\sqrt{\delta_j}
\qquad y\in K.
\end{align*}
The definition of $H_4$-center and Cauchy-Schwarz inequality implies
\begin{align*}
d_{W_1}^{g(s_j)}
\bigl(\nu_{x_{i(j)},t_{i(j)};s_j},\delta_{\psi_j(y)}\bigr)
\le d_{g(s_j)}(\psi_j(y),z_j)+\sqrt{H_4\delta_j}
\le C_K\sqrt{\delta_j}.
\end{align*}
Set $b_j=t_{i(j)}-\varepsilon^2\tau_{i(j)}<s_j$, the monotonicity of $d_{W_1}^{g_t}$ implies 
\[
 d_{W_1}^{g(b_j)}
 \bigl(\nu_{x_{i(j)},t_{i(j)};b_j},
       \nu_{\psi_j(y),s_j;b_j}\bigr)
 \le C_K\sqrt{\delta_j}<\varepsilon\sqrt{\tau_{i(j)}}.
\]
Thus $(\psi_j(y),s_j)$ lies in the neighborhood in
\eqref{eq:typeI-exterior-bad}, so $\psi_j(K) \cap \mathcal F=\varnothing$.
The Schwarz estimate and \cite[Lemma 2.10]{CHM25} give
\begin{align*}
d_{\omega_\Sigma}\bigl(p(\psi_j(y)),p(x_{i(j)})\bigr)
\le d_{\omega_\Sigma}\bigl(p(\psi_j(y)),p(z_j)\bigr)
   +d_{\omega_\Sigma}\bigl(p(z_j),p(x_{i(j)})\bigr)
\le C_K\sqrt{\delta_j}
\qquad y\in K.
\end{align*}
It follows that $\psi_j(K)\subset V\setminus F_q$ for large $j$. At any singular point of $X'$, choose a small link and an
annular neighborhood whose closure lies in $X'_{\mathrm{reg}}$, Lemma~\ref{lem:typeI-exterior-link} shows that $X'$ is smooth. We know $X'$ is a nontrivial shrinker. By \cite[Theorem 1.1]{LW26}, $X'$ contains a
smooth compact holomorphic curve $C$. For sufficiently large $j$,
$\psi_j(C)\subset V$. Set $\beta_j=(\psi_j)_*[C]\in H_2(V;\mathbb Z)$ and
$n_j=\langle c_1(M)|_V,\beta_j\rangle\in\mathbb Z$. Recall $[\omega(s_j)|_V]= (T-s_j)c_1(M)|_V$, so
\[
{(T-s_j)}{\delta_j}^{-1}n_j
=\int_{\beta_j}\delta_j^{-1}\omega(s_j)
=\int_C\delta_j^{-1}\psi_j^*\omega(s_j)
\longrightarrow \int_C\omega'>0.
\]
For large $j$, this implies $n_j\ge1$, but ${(T-s_j)}{\delta_j}^{-1}=1+\lambda_j^{-2}\to \infty$, contradicting the finite limit.
\end{proof}

\subsection{Proof of the main theorem}
\begin{proof}[Proof of the Theorem 1.1]
Suppose otherwise, choose
$\tau_i=T-t_i\searrow0$ with
$\tau_i|\operatorname{Rm}(g(t_i))|(x_i)\to\infty$.
By Proposition~\ref{prop:typeI-exterior}, with $\varepsilon=1$,
the neighborhoods $P^{*-}(x_i,t_i;\sqrt{\tau_i})\cap (\mathcal F\times[0,T))\neq \varnothing$ for large $i$.
After passing to a subsequence, they meet one fixed singular fiber $F_q$.
We consider the rescaled flow $g_i(t)=\tau_i^{-1}g(T+\tau_i t)$,
choose
\begin{equation}\label{eq:typeI-fiber-intersection}
 (y_i,\sigma_i)\in
 (F_q\times[-2,-1])\cap P^{*-}_{g_i}(x_i,-1;1).
\end{equation}
By Theorem \ref{prop:visibility-bccd}, the unique tangent flow $X$ corresponding to any conjugate heat kernel based at $(x_0,T)$ is the BCCD shrinker. Write its canonical flow as $(X,h(t),J_X)$ and fix a compact set $K$ of the central fiber in $X$. \cite[Corollary 2.4]{CCD} and smooth convergence imply that we can obtain $\psi^*_iJ_M$ curves $C_1$ and $C_2$ (deform the central fiber of BCCD shrinker),
their images of $\psi_i$ lie in $V$ by \cite[Lemma 2.11]{CHM25}.
They are distinct, since their local classes are independent,
so $F_q\subset\psi_i(K)$. Because $X$ has bounded curvature, the smooth convergence implies that there exists $K_0<\infty$ such that, for every fixed $R>0$
and sufficiently large $i$,
\begin{align}\label{eq:typeI-captured-cylinders}
    |\operatorname{Rm}(g_i(t))|(x)\le K_0
\quad\text{on } B_{g_i(s)}(y,R)\times[-4,-1],
\end{align}
for every $y\in F_q$ and $s\in[-3,-1]$. By \eqref{eq:typeI-fiber-intersection}, $d_{W_1}^{g_i(-2)}
\bigl(\nu_{x_i,-1;-2},\nu_{y_i,\sigma_i;-2}\bigr)<1$. The definition of $P^*$ gives $(x_i,-1)\in
P^*_{g_i}(y_i,\sigma_i;1,-(\sigma_i+2),-1-\sigma_i)$. Since $\sigma_i+2,-1-\sigma_i\in[0,1]$,
\cite[Corollary 9.6(b)]{Bam20a}, 
and \eqref{eq:typeI-captured-cylinders} give a fixed
$R<\infty$ such that, for all large $i$,
\[
P^*_{g_i}\bigl(y_i,\sigma_i;1,-(\sigma_i+2),-1-\sigma_i\bigr)
\subset B_{g_i(\sigma_i)}(y_i,R)\times[-2,-1].
\]
Hence $|\operatorname{Rm}(g_i(-1))|(x_i)\le K_0$, contradicting the choice of $(x_i,t_i)$.
\end{proof}
Now we identify the tangent flow on any singular fiber $F_q$ for $q\in \Delta$ in the sense of \cite{EMT}.

\begin{proof}[proof of Corollary 1.2]
We know $M$ has Type I curvature bound and set $g_i(t):=\tau_i^{-1}g(T+\tau_i t)$, then by \cite[Theorems 1.4]{EMT}
\begin{align*}
    \bigl(M, g_i(t), J_M, x_0\bigr)_{t\in[-\tau^{-1}_iT,0)} 
\xrightarrow{\text{pointed-}C^\infty\text{-Cheeger--Gromov}} 
\bigl(X, g_\infty(t), J_\infty, x_\infty\bigr)_{t\in(-\infty,0)},
\end{align*}
$X$ is a nonflat Kähler-Ricci shrinker surface with bounded curvature. Furthermore, the smooth convergence together with the divergence of the rescaled volume
forces the limit $X$ to be noncompact. Due to the volume of $M$ at singular time is collapsed, \cite[Theorem 1.1]{LW26} and \cite[Theorem B]{CCD} imply $X$ is the shrinking cylinder or BCCD shrinker.

Suppose that $X$ is the cylinder. After normalization, we set $o:=x_\infty=(p_0,0)$. At time $-1$, let
$C=\mathbb P^1\times\{o\}$ be the holomorphic sphere through $o$ with $C^2=0$. On a fixed neighborhood of $C$, choose pointed
convergence maps $\psi_i$ with $\psi_i(o)=x_0$. By
\cite[Corollary 2.3]{CCD}, $C$ deforms to a
$\psi_i^*J$-holomorphic sphere $C_i$ through $o$, homologous
to $C$ in this neighborhood, so $C_i$ is embedded for large $i$.
Thus $D_i=\psi_i(C_i)$ satisfies $x_0\in D_i$, $D_i\simeq\mathbb P^1$ and $D_i^2=0$. By the same arguments in Theorem \ref{prop:visibility-bccd}, we know $D_i\subset F_q$.

Since $D_i$ is an irreducible compact holomorphic curve contained
in $F_q=C_1+C_2$, it must coincide with one of the irreducible
components $C_1$ or $C_2$. As both components have negative
self-intersection, this contradicts $D_i^2=0$. This excludes the shrinking cylinder.
\end{proof}

\begin{proof}[Proof of Corollary 1.3]
Set $\tau=T-t$ and $\widehat g_t=\tau^{-1}g(t)$.
The Type I estimate and Perelman's noncollapsing theorem give uniformly bounded geometry:
\[
|\operatorname{Rm}(\widehat g_t)|\le K,
\qquad
\operatorname{inj}(M,\widehat g_t)\ge\iota>0.
\]
for any $t \in [0,T)$. Choose $r_0,b>0$ with $r_0<\iota/2$ and $br_0^2<1/2$.
Hessian comparison gives
\[
\nabla^2 f_p\ge(1-br^2)\widehat g_t,
\qquad
|\nabla f_p|\le r
\quad\text{on }B_{\widehat g_t}(p,r),
\]
for $f_p=\tfrac12d_{\widehat g_t}(p,\cdot)^2$ with any $p\in M$ and $0<r\le r_0$.

Let $E$ be an irreducible component of some fiber.
It is a smooth closed holomorphic curve, hence a minimal
surface. The area formula \ref{eq:local-class} gives $\operatorname{Area}_{\widehat g_t}(E)\le4\pi$.
We set $g_E:={\widehat g_t}|_E$. If $\operatorname{diam}_{\widehat g_t}E<r_0$, then
$E\subset B_{\widehat g_t}(p,r_0)$ for any $p\in E$.
Minimality gives
\begin{align*}
  \Delta_{g_E} f_p
=\operatorname{tr}_{g_E}\nabla_M^2f_p+2\langle\nabla^M f_p,\vec H\rangle
\ge 2(1-br_0^2)>0,  
\end{align*}
contradicting $\textstyle\int_E\Delta_{g_E} f_p=0$. Thus $\operatorname{diam}^{ext}_{\widehat g_t}E\ge r_0$. Due to $d_{\widehat g_t}(x,y)\leq d_{g_E}(x,y)$, we know $\operatorname{diam}^{int}_{\widehat g_t}E\ge r_0$.

Fix $p\in E$ and set $V(r)=\operatorname{Area}_{g_E}(B_{g_E}(p,r))$, where $B_{g_E}(p,r)$ is an intrinsic ball.
Since $B_{g_E}(p,r)\subset B_{\widehat g_t}(p,r)$, the divergence
and coarea formulas give
\[
2(1-br^2)V(r)
\le \int_{B_{g_E}(p,r)}\Delta_{g_E} f_p=
\int_{\partial B_E(p,r)}
\langle\nabla_{g_E} f_p,\nu\rangle\,d\mathcal{H}^1
\le r\, \mathcal{H}^1(\partial B_{g_E}(p,r))
=rV'(r),
\]
for almost every $r\le r_0$. As $V(r)/r^2\to\pi$ when $r\searrow0$, integration yields $V(r)\ge\pi r^2e^{-br^2}$. Since $\operatorname{Area}_{\widehat g_t}(E)\le4\pi$, the standard covering argument along a minimal geodesic in $E$ implies $\operatorname{diam}^{\mathrm{int}}_{\widehat g_t}E
\le C_0$.

A singular fiber consists of two such components meeting
at one point. Joining paths through this point gives $\operatorname{diam}^{\mathrm{int}}_{\widehat g_t}F_q
\le2C_0$. The regular-fiber case follows from the estimate for $E$. Rescaling proves the claim.
\end{proof}

\bibliographystyle{alpha}
\bibliography{ref.bib}

@article{TosattiZhang2018FiniteTimeCollapsing,
  author  = {Tosatti, Valentino and Zhang, Yuguang},
  title   = {Finite time collapsing of the {K\"ahler--Ricci} flow on threefolds},
  journal = {Ann. Sc. Norm. Super. Pisa Cl. Sci. (5)},
  volume  = {18},
  number  = {1},
  pages   = {105--118},
  year    = {2018},
  doi     = {10.2422/2036-2145.201508_003},
  eprint  = {1507.08397},
  archivePrefix = {arXiv},
  primaryClass  = {math.DG}
}

@article{JianSongTian2023FiniteTimeSingularities,
  author  = {Jian, Wangjian and Song, Jian and Tian, Gang},
  title   = {Finite time singularities of the {K\"ahler--Ricci} flow},
  year    = {2023},
  journal={arXiv preprint arXiv:2310.07945},
  primaryClass  = {math.DG}
}

@article{BCCD,
  author  = {Bamler, Richard H. and Cifarelli, Charles and Conlon, Ronan J. and Deruelle, Alix},
  title   = {A new complete two-dimensional shrinking gradient {K\"ahler--Ricci} soliton},
  journal = {Geom. Funct. Anal.},
  volume  = {34},
  number  = {2},
  pages   = {377--392},
  year    = {2024},
  doi     = {10.1007/s00039-024-00668-9},
  eprint  = {2206.10785},
  archivePrefix = {arXiv},
  primaryClass  = {math.DG}
}

@article{CCD,
  author  = {Cifarelli, Charles and Conlon, Ronan J. and Deruelle, Alix},
  title   = {On finite time {Type I} singularities of the {K\"ahler--Ricci} flow on compact {K\"ahler} surfaces},
  journal = {J. Eur. Math. Soc.},
  volume  = {28},
  number  = {2},
  pages   = {463--504},
  year    = {2024},
  doi     = {10.4171/JEMS/1485},
  eprint  = {2203.04380},
  archivePrefix = {arXiv},
  primaryClass  = {math.DG}
}

@article{Bam20a,
  title={Entropy and heat kernel bounds on a Ricci flow background},
  author={Bamler, Richard H},
  journal={arXiv preprint arXiv:2008.07093},
  year={2020}
}

@article{Bam23,
  title={Compactness theory of the space of super Ricci flows},
  author={Bamler, Richard H},
  journal={Inventiones mathematicae},
  volume={233},
  number={3},
  pages={1121--1277},
  year={2023},
  publisher={Springer}
}

@article{Bam20b,
  title={Structure theory of non-collapsed limits of Ricci flows},
  author={Bamler, Richard H},
  journal={arXiv preprint arXiv:2009.03243},
  year={2020}
}

@article{CHM25,
  title={Non-collapsed finite time singularities of the Ricci flow on compact Kähler surfaces are of Type {I}},
  author={Conlon, Ronan J and Hallgren, Max and Ma, Zilu},
  journal={arXiv preprint arXiv:2502.19804},
  year={2025}
}

@article{HJ23,
  title={Tangent flows of K{\"a}hler metric flows},
  author={Hallgren, Max and Jian, Wangjian},
  journal={Journal f{\"u}r die reine und angewandte Mathematik (Crelles Journal)},
  volume={2023},
  number={805},
  pages={143--184},
  year={2023},
  publisher={De Gruyter}
}

@article{LW26,
  title={On Kähler Ricci shrinker surfaces},
  author={Li, Yu and Wang, Bing},
  journal={Acta Mathematica},
  volume={236},
  number={1},
  pages={1--50},
  year={2026},
  publisher={International Press of Boston, Inc. Somerville, MA 02143, USA}
}

@article{simon2014introduction,
  title={Introduction to geometric measure theory},
  author={Simon, Leon},
  journal={Tsinghua lectures},
  volume={2},
  number={2},
  pages={3--1},
  year={2014}
}

@article{hamilton1982three,
  title={Three-manifolds with positive Ricci curvature},
  author={Hamilton, Richard S},
  journal={Journal of Differential geometry},
  volume={17},
  number={2},
  pages={255--306},
  year={1982},
  publisher={Lehigh University}
}

@article {EMT,
    AUTHOR = {Enders, Joerg and M\"uller, Reto and Topping, Peter M.},
     TITLE = {On type-{I} singularities in {R}icci flow},
   JOURNAL = {Comm. Anal. Geom.},
  FJOURNAL = {Communications in Analysis and Geometry},
    VOLUME = {19},
      YEAR = {2011},
    NUMBER = {5},
     PAGES = {905--922},
      ISSN = {1019-8385,1944-9992},
   MRCLASS = {53C44},
  MRNUMBER = {2886712},
MRREVIEWER = {Juan-Ru\ Gu},
       DOI = {10.4310/CAG.2011.v19.n5.a4},
       URL = {https://doi.org/10.4310/CAG.2011.v19.n5.a4},
}

@article{Naber,
  title={Noncompact shrinking four solitons with nonnegative curvature},
  author={Naber, Aaron},
  journal={Journal f{\"u}r die reine und angewandte Mathematik (Crelles Journal)},
  volume={2010},
  number={645},
  pages={125},
  year={2010},
  publisher={Walter de Gruyter GmbH}
}

@article{Tia90,
  author  = {Tian, Gang},
  title   = {On {Calabi}'s conjecture for complex surfaces with
             positive first {Chern} class},
  journal = {Inventiones Mathematicae},
  volume  = {101},
  number  = {1},
  pages   = {101--172},
  year    = {1990},
  doi     = {10.1007/BF01231499}
}

@incollection{Koi90,
  author    = {Koiso, Norihito},
  title     = {On rotationally symmetric {Hamilton}'s equation for
               {K\"ahler--Einstein} metrics},
  booktitle = {Recent Topics in Differential and Analytic Geometry},
  editor    = {Ochiai, Takushiro},
  series    = {Advanced Studies in Pure Mathematics},
  volume    = {18-I},
  pages     = {327--337},
  publisher = {Academic Press},
  address   = {Boston, MA},
  year      = {1990},
  doi       = {10.1016/B978-0-12-001018-9.50015-4}
}

@article{WZ04,
  author  = {Wang, Xu-Jia and Zhu, Xiaohua},
  title   = {{K\"ahler--Ricci} solitons on toric manifolds with
             positive first {Chern} class},
  journal = {Advances in Mathematics},
  volume  = {188},
  number  = {1},
  pages   = {87--103},
  year    = {2004},
  doi     = {10.1016/j.aim.2003.09.009}
}

@article{ST08,
  title={Bounding scalar curvature and diameter along the K{\"a}hler Ricci flow (after Perelman)},
  author={Sesum, Natasa and Tian, Gang},
  journal={Journal of the Institute of Mathematics of Jussieu},
  volume={7},
  number={3},
  pages={575--587},
  year={2008},
  publisher={Cambridge University Press}
}

@article{TZ07,
  author  = {Tian, Gang and Zhu, Xiaohua},
  title   = {Convergence of {K\"ahler--Ricci} flow},
  journal = {Journal of the American Mathematical Society},
  volume  = {20},
  number  = {3},
  pages   = {675--699},
  year    = {2007},
  doi     = {10.1090/S0894-0347-06-00552-2}
}

@incollection{BM87,
  author    = {Bando, Shigetoshi and Mabuchi, Toshiki},
  title     = {Uniqueness of {Einstein} {K\"ahler} metrics modulo
               connected group actions},
  booktitle = {Algebraic Geometry, Sendai, 1985},
  series    = {Advanced Studies in Pure Mathematics},
  volume    = {10},
  pages     = {11--40},
  publisher = {North-Holland},
  address   = {Amsterdam},
  year      = {1987},
  doi       = {10.2969/aspm/01010011}
}

@article{TZ00,
  author  = {Tian, Gang and Zhu, Xiaohua},
  title   = {Uniqueness of {K\"ahler--Ricci} solitons},
  journal = {Acta Mathematica},
  volume  = {184},
  number  = {2},
  pages   = {271--305},
  year    = {2000},
  doi     = {10.1007/BF02392630}
}

@incollection{SWL,
  title={An introduction to the K{\"a}hler--Ricci flow},
  author={Song, Jian and Weinkove, Ben},
  booktitle={An introduction to the K{\"a}hler-Ricci flow},
  pages={89--188},
  year={2013},
  publisher={Springer}
}

@article{MM15,
  author = {Carlo Mantegazza and Reto M{\"u}ller},
  title = {Perelman's entropy functional at Type I singularities of the Ricci flow},
  journal = {Journal f{\"u}r die reine und angewandte Mathematik},
  volume = {703},
  year = {2015},
  pages = {173--199},
  mrnumber = {3353546},
  doi = {10.1515/crelle-2013-0039}
}

@article{MT23,
  author  = {Miao, M. and Tian, G.},
  title   = {A Note on Kähler-Ricci Flow on Fano Threefolds},
  journal = {Peking Math. J.},
  year    = {2023}
}

@article{LTZ24,
  author    = {Li, Y. and Tian, G. and Zhu, X.},
  title     = {Singular limits of Kähler-Ricci flow on Fano G-manifolds},
  journal   = {Amer. J. Math.},
  volume    = {146},
  number    = {6},
  pages     = {1651--1690},
  year      = {2024},
  mrnumber  = {MR4855863}
}

@article{XZ26,
  title={Compactness and Rigidity of Complete K$\backslash$" ahler Ricci Shrinkers},
  author={Xu, Tongxin and Zhang, Zhenlei},
  journal={arXiv preprint arXiv:2608.10953},
  year={2026}
}

@article{HZ26,
      title={Singular K\"ahler--Ricci shrinkers and polarized Fano fibrations},
      author={Max Hallgren and Junsheng Zhang},
      year={2026},
      eprint={2605.25213},
      archivePrefix={arXiv},
      primaryClass={math.DG},
      url={https://doi.org/10.48550/arXiv.2605.25213}
}

@article{SZ,
  title={Kähler-Ricci shrinkers and Fano fibrations},
  author={Sun, Song and Zhang, Junsheng},
  journal={arXiv preprint arXiv:2410.09661},
  year={2024}
}

@article{CT,
  title={K{\"a}hler currents and null loci},
  author={Collins, Tristan C and Tosatti, Valentino},
  journal={Inventiones mathematicae},
  volume={202},
  number={3},
  pages={1167--1198},
  year={2015},
  publisher={Springer}
}

@book{HR,
  author    = {Morris W. Hirsch},
  title     = {Differential Topology},
  series    = {Graduate Texts in Mathematics},
  volume    = {33},
  publisher = {Springer-Verlag},
  address   = {New York-Heidelberg},
  year      = {1976},
  mrnumber  = {MR0448362}
}

@misc{De,
  author = {Jean-Pierre Demailly},
  title  = {Complex Analytic and Differential Geometry},
  year   = {2012},
  note   = {Online book, version of June 21, 2012},
  url    = {https://www-fourier.ujf-grenoble.fr/~demailly/source_files/analgeom/agbook.pdf}
}

@book{Beauville,
  title={Complex algebraic surfaces},
  author={Beauville, Arnaud},
  number={34},
  year={1996},
  publisher={Cambridge University Press}
}

@article{XZ,
  title={Finite Time Singularities of Collapsing K$\backslash$" ahler Ricci Flow on Ruled Surfaces},
  author={Xu, Tongxin and Zhang, Zhenlei},
  journal={arXiv preprint arXiv:2609.01442},
  year={2026}
}

@book{Federer,
  author    = {Federer, Herbert},
  title     = {Geometric Measure Theory},
  series    = {Die Grundlehren der mathematischen Wissenschaften},
  volume    = {153},
  publisher = {Springer-Verlag},
  address   = {Berlin},
  year      = {1969}
}

@article{King,
  author  = {King, James R.},
  title   = {The currents defined by analytic varieties},
  journal = {Acta Math.},
  volume  = {127},
  year    = {1971},
  pages   = {185--220},
  doi     = {10.1007/BF02392053}
}

@article{TehYang,
  author  = {Jyh-Haur Teh and Chin-Jui Yang},
  title   = {Real rectifiable currents, holomorphic chains and
             algebraic cycles},
  journal = {Complex Manifolds},
  volume  = {8},
  number  = {1},
  pages   = {274--285},
  year    = {2021},
  doi     = {10.1515/coma-2020-0119}
}

@article{SW,
  author = {Song, Jian and Weinkove, Ben},
  title = {Contracting exceptional divisors by the Kähler-Ricci flow},
  journal = {Duke Math. J.},
  volume = {162},
  number = {2},
  year = {2013},
  pages = {367--415},
  mrnumber = {3018957},
  doi = {10.1215/00127094-1962881},
}

@article{JS,
  title={Finite-Time Singularities of the K\"ahler--Ricci Flow on Fano Bundles},
  author={Wangjian Jian and Jian Song},
  year={2026},
  eprint={2609.02878},
  archivePrefix={arXiv},
  primaryClass={math.DG}
}

@article{CCHZ,
  title={Finite time singularities of the Ricci flow on compact K\"ahler surfaces are of Type I},
  author={Charles Cifarelli and Ronan Conlon and Max Hallgren and Junsheng Zhang},
  year={2026},
  eprint={2609.16733},
  archivePrefix={arXiv},
  primaryClass={math.DG},
  note={Submitted on 15 Sep 2026}
}

@article{CW12,
  title={The K\"ahler Ricci flow on Fano surfaces (I)},
  author={Xiuxiong Chen and Bing Wang},
  journal={Math. Z.},
  volume={270},
  number={1--2},
  year={2012},
  pages={577--587}
}

@article{LZ26,
  title={Gromov-Hausdorff Limits of Noncollapsed K\"ahler-Ricci Flows and the Geometry of Ricci Shrinkers},
  author={Y Li and J Zhang},
  year={2026},
  eprint={2607.25644},
  archivePrefix={arXiv},
  primaryClass={math.DG}
}

@article{HJST,
  author  = {Hallgren, Max and Jian, Wangjian and Song, Jian and Tian, Gang},
  title   = {Geometric regularity of blow-up limits of the
             {K{\"a}hler--R}icci flow},
  journal = {Geometric and Functional Analysis},
  volume  = {34},
  number  = {6},
  pages   = {1899--1972},
  year    = {2024},
  doi     = {10.1007/s00039-024-00694-7}
}

@article{chen2026k,
  title={K$\backslash$" ahler-Ricci Tangent Flows in the Analytic Minimal Model Program},
  author={Chen, Longteng and Hallgren, Max and Lavoyer, Lucas},
  journal={arXiv preprint arXiv:2608.19152},
  year={2026}
}
\section*{Author Information}
\noindent\textbf{Tongxin Xu}$^*$\\
School of Mathematical Sciences, Capital Normal University\\
Email: 2250501032@cnu.edu.cn
\vspace{0.5em}

\noindent\textbf{Zhenlei Zhang}$^\dagger$\\
School of Mathematical Sciences, Capital Normal University\\
Email: zhleigo@aliyun.com\\
\end{document}